\documentclass[reqno, 11pt]{amsart}
\usepackage[utf8]{inputenc}
\usepackage[T1]{fontenc}
\usepackage[english]{babel}
\usepackage{fullpage}
\usepackage{amssymb,amsmath,amscd,mathrsfs,graphicx,amsthm,stmaryrd,amsfonts,amsxtra}
\usepackage{bm}
\usepackage{rotating,enumerate,float}
\usepackage{tikz-cd, tikz, verbatim,mathtools,multirow, xcolor}
\usepackage{hyperref}
\usepackage{cleveref}

\usepackage[sorting=nyc,doi=false,isbn=false,url=false,giveninits=true,maxbibnames=99]{biblatex}

\usepackage[left=3.4cm,right=3.4cm,bottom=3.4cm,top=3.4cm]{geometry}

\tikzstyle{nodal}=[circle,draw,fill=black,inner sep=0pt, minimum width=4pt]

 \DeclareFontFamily{U}{wncy}{}
     \DeclareFontShape{U}{wncy}{m}{n}{<->wncyr10}{}
     \DeclareSymbolFont{mcy}{U}{wncy}{m}{n}
     \DeclareMathSymbol{\Sh}{\mathord}{mcy}{"58} 

 \DeclareSortingScheme{nyc}{
   \sort{
     \field{presort}
   }
   \sort[final]{
     \field{sortkey}
   }
   \sort{
     \field{sortname}
     \field{author}
     \field{editor}
     \field{translator}
     \field{sorttitle}
     \field{title}
   }
   \sort{
     \field{sortyear}
     \field{year}
   }
   \sort{\citeorder}
 }

\newtheorem{theorem}{Theorem}[section]
\newtheorem{proposition}[theorem]{Proposition}
\newtheorem{lemma}[theorem]{Lemma}

\theoremstyle{definition}

 \newtheorem*{acknowledgments*}{Acknowledgments}

\theoremstyle{remark}
\newtheorem{remark}[theorem]{Remark}

\newcommand{\Gal}{\mathrm{Gal}}
\newcommand{\Hom}{\mathrm{Hom}}

\newcommand{\Aut}{\mathrm{Aut}}

\newcommand{\Br}{\mathrm{Br}}

\newcommand{\PP}{\mathbb{P}}
\newcommand{\CC}{\mathbb{C}}
\newcommand{\QQ}{\mathbb{Q}}
\newcommand{\ZZ}{\mathbb{Z}}

\newcommand{\cO}{\mathcal{O}}

\newcommand{\cS}{\mathcal{S}}

\newcommand{\Pic}{\mathrm{Pic}}
\DeclareMathOperator{\rk}{rk}
\newcommand{\disc}{\mathrm{disc}}

\newcommand{\Eff}{\mathrm{Eff}}
\newcommand{\RH}{\mathrm{H}}

\newcommand{\piet}{\pi_1^{\mathrm{\acute{e}t}}}
\newcommand{\pitop}{\pi_1^{\mathrm{top}}}

\newcommand{\ov}{\overline}

\newcommand{\Gi}[1]{\textcolor{blue}{(Gi: #1)}}

\title{Non-thin rational points for elliptic K3 surfaces}

\author{Damián Gvirtz-Chen}
\address{School of Mathematics and Statistics\\ University of Glasgow \\ University Place\\ Glasgow G12 8QQ \\ United Kingdom}
\email{damian.gvirtz@glasgow.ac.uk}

\author{Giacomo Mezzedimi}
\address{Mathematisches Institut \\ Universität Bonn \\ Endenicher Allee 60 \\ 53115 Bonn \\ Germany}
\email{mezzedim@math.uni-bonn.de}

\date{\today}

\begin{document}

\maketitle

\begin{abstract}
We prove that elliptic K3 surfaces over a number field which admit a second elliptic fibration satisfy the potential Hilbert property. Equivalently, the set of their rational points is not thin after a finite extension of the base field.

Furthermore, we classify those families of elliptic K3 surfaces over an algebraically closed field which do not admit a second elliptic fibration.
\end{abstract}

\section{Introduction}

The purpose of this note is to prove the following theorem:

\begin{theorem}\label{thm:main}
Let $X$ be a K3 surface over a number field $k$ admitting two distinct elliptic fibrations. Then, after a finite field extension $k'/k$, the set of rational points $X(k')$ is not thin. 
\end{theorem}
The degree $[k':k]$ can be uniformly bounded, see \Cref{rem:bounded}.

Varieties with a non-thin set of rational points after a finite extension of the base field are said to satisfy the \emph{potential Hilbert Property} (see \cite[\S3.1]{serre.topics.galois.theory} for an overview of thin sets and the Hilbert Property). A widely open conjecture states that all K3 surfaces over number fields satisfy the potential Hilbert Property \cite[Conjecture~1.7]{CDJLZ} and \Cref{thm:main} answers this conjecture positively in our situation. 

Special cases of \Cref{thm:main} are known. Diagonal quartic surfaces were treated by Corvaja and Zannier in \cite[Theorem~1.6]{corvaja.zannier.hilbert}, and product Kummer surfaces by Demeio in \cite[Proposition~4.4]{demeio.elliptic.fibrations}. Further progress was made in \cite{gvirtz.mezzedimi.hilbert}, where we showed that the potential Hilbert Property holds for K3 surfaces $X$ which (a) cover an Enriques surface, or (b) admit two elliptic fibrations and have Picard rank at most $9$, or (c) have Picard rank at least $6$ except for a finite list $\cS$ of excluded geometric Picard lattices $\Pic(X_{\ov k})$. The list $\cS$ in (c) is not explicitly known and expected to be very large.

In all the mentioned known cases, the so-called \emph{over-exceptional lattice} $E'(X_{\ov k})$, introduced in \cite{gvirtz.mezzedimi.hilbert}, embeds primitively into $\Pic(X_{\ov k})$ and has small rank, from which the potential Hilbert Property follows. The general statement of \Cref{thm:main} has so far been out of reach due to a lack of understanding of the imprimitive possibility. As we are going to show, the conditions (a)-(c) can be leveraged by a descent argument for the Hilbert Property along K3 covers and a finiteness theorem for universal covers of open K3 surfaces, to yield the full \Cref{thm:main}.

Moreover, the recently completed classification of K3 surfaces of zero entropy by Yu \cite{yu.entropy} and the second author and Brandhorst \cite{brandhorst.mezzedimi.borcherds} allows us to give a complete list of those K3 surfaces with only one elliptic fibration, comprising $92$ families with geometric Picard ranks between $3$ and $12$ and one class of families in Picard rank $2$. We determine this list in Section \ref{sec:unique.elliptic.fibration}. Another formulation of our result is thus that elliptic K3 surfaces which do not fall into one of the $93$ classes (e.g.\ whenever the geometric Picard rank is greater than $12$) satisfy the potential Hilbert Property. 

As far as the potential Hilbert Property is concerned, the present work conclusively exhausts the method started by Corvaja and Zannier. To prove the potential Hilbert Property for \emph{every} elliptic K3 surface, all that remains are the $93$ classes. It would be an interesting open problem to deal with these cases, as doing so will likely lead to significant input from new methods.




\begin{acknowledgments*}
We warmly thank Reinder Meinsma for the many useful conversations about Jacobians of elliptic K3 surfaces.
\end{acknowledgments*}

\section{Elliptic K3 surfaces}

In the following let $k$ be a field of characteristic $0$, and $\ov k$ an algebraic closure of $k$. 
A \emph{K3 surface} over $k$ is a smooth, projective, geometrically integral surface with trivial canonical bundle and $H^1(X,\cO_X)=0$. A K3 surface $X$ is \emph{elliptic} if it is endowed with an \emph{elliptic fibration}, i.e. a proper and flat morphism $\pi: X\to \PP^1$ such that the general fiber is a smooth curve of genus $1$. Every elliptic fibration on $X$ arises from a pencil $|F|$, where $F\in \Pic(X)$ is a primitive, nef, isotropic vector represented by a curve of arithmetic genus $1$.
We recall a version of the Shioda-Tate formula that will be useful in the following. We denote by $\rho(X)\coloneqq \rk \Pic(X)$ the \emph{Picard rank} of $X$, and if $N$ is a negative definite lattice, $N_{root}$ is the sublattice of $N$ spanned by the $(-2)$-roots.

\begin{proposition}[Shioda-Tate formula] \label{prop:shioda.tate}
Let $X$ be a K3 surface over $\ov k$ and $|F|$ an elliptic fibration on $X$. The stabilizer of $F\in \Pic(X)$ in $\Aut(X)$ is infinite if and only if
\begin{equation*}
\rho(X)-2 - \rk((F^\perp/\langle F \rangle)_{root}) > 0.
\end{equation*}
\end{proposition}
\begin{proof}
Denote by $G$ the stabilizer in $\Aut(X)$ of $F\in \Pic(X)$. Since the elliptic fibration $|F|$ has at least three singular fibers (see e.g. \cite[Last~paragraph~of~§1.1]{mezzedimi.entropy}), the action of $G$ on the base $\PP^1$ is finite, hence $G$ is isomorphic up to a finite group to the group of automorphisms of the generic fiber $F_\eta$.

If $|F|$ has a section, then $G$ is isomorphic to the Mordell--Weil group of $|F|$ up to a finite group, so the claim follows by the usual Shioda-Tate formula (see e.g. \cite[Corollary~6.7]{schuett.shioda.mordell.weil}).

If instead $|F|$ has no section, the associated Jacobian fibration $|JF|: JX \to \PP^1$ is a K3 surface of the same Picard rank \cite[Proposition~11.4.5]{huybrechts.k3.surfaces}, and it has the same singular fibers as $|F|$ \cite[Theorem~4.3.20]{CossecDolgachevLiedtke}. It is well-known that the group of automorphisms of $F_\eta$ is isomorphic up to a finite group to the group of $\ov k (\PP^1)$-points of the Jacobian $JF_\eta$, which is in turn isomorphic to the Mordell--Weil group of $|JF|$.
\end{proof}

Let $X$ be a K3 surface over $k$ admitting an elliptic fibration with a section. The \emph{Tate-Shafarevich group} $\Sh(X)$ parametrizes isomorphism classes of elliptic K3 surfaces $Y$ over $k$ whose Jacobian is isomorphic to $X$. There is a natural isomorphism $\Sh(X) \cong \Br(X)/\Br(k)$ \cite[Proposition~(4.5)]{grothendieck.dixIII}. Applied to $X_{\ov k}$, this yields \[\Sh(X_{\ov k})\cong\Br(X_{\ov k})\cong(\QQ/\ZZ)^{22-\rho(X_{\ov k})},\] where the first isomorphism is compatible with the $\Gal(\ov k/k)$-action. In particular $\Sh(X_{\ov k})$ is a divisible group.

For any elliptic K3 surface $Y$ with $[Y]\in \Sh(X)$ and any integer $n\ge 0$, one can check that the higher Jacobian $\mathrm{Jac}^n(Y)$ has class $[\mathrm{Jac}^n(Y)] = n[Y]$ in $\Sh(X)$ \cite[Remark~11.5.2]{huybrechts.k3.surfaces}. By extending the natural multiplication-by-$n$ morphism between the generic fibers, one obtains a rational dominant morphism $Y \cong \mathrm{Jac}^1(Y) \dashrightarrow \mathrm{Jac}^n(Y)$ that is multiplication by $n$ on the smooth fibers.

\section{Proof of \Cref{thm:main}}
Let $X$ be a K3 surface over a subfield $k\subset\CC$ admitting two distinct elliptic fibrations. Recall that the over-exceptional lattice was defined in \cite[Definition 5.1]{gvirtz.mezzedimi.hilbert} as the sublattice $E'(X)$ of $\Pic(X)$ that is generated by the classes of curves orthogonal to all elliptic fibrations on $X$. The presence of two distinct elliptic fibrations implies that $E'(X)$ is a root lattice, i.e.\ an even, negative definite lattice that admits a basis
given by vectors of norm $-2$. Denote by $\Delta\subset X$ the union of $(-2)$-curves whose class lies in $E'(X)$.

There exists a finite extension $L/k$ such that $\Pic(X_L)\simeq \Pic(X_{\ov k})\simeq \Pic(X_{\CC})$ \cite[XIII, Théorème~5.1]{sga6}. Replacing $k$ with $L$, it suffices to prove the potential Hilbert Property for $X$ when the full geometric Picard lattice is realised over the base field. From now on, we will assume that this is the case. In particular
\[E'(X)\simeq E'(X_{\ov k}).\]

We will need an easy-to-prove bound for the rank of the over-exceptional lattice. 
\begin{lemma}\label{lem:bound}
 The rank of the over-exceptional lattice $E'(X)$ is at most $17$.
\end{lemma}
\begin{proof}
 Let $F$ be the fibre class of an elliptic fibration $|F|:X\to\PP^1$. Since $X$ admits another elliptic fibration $|F'|$, which must satisfy $F'\cdot F>0$, $F$ cannot lie in $E'(X)$. It follows that $E'(X)$ embeds into $(F^\perp/\langle F\rangle)_{root}$.
 
 The Shioda--Tate formula implies that $\rk E'(X)\leq \rho(X)-2$ with strict inequality if the Mordell--Weil group of $|F|$ is infinite. The rank bound $\rk E'(X)\leq 17$ follows immediately if $\rho(X)\leq 19$, and also if $\rho(X)=20$, using that K3 surfaces with maximal Picard rank always admit an elliptic fibration with positive Mordell--Weil rank (cf. \cite[Theorem~5]{shioda.inose}).
\end{proof}

\begin{proposition}\label{prop:piet}
 The fundamental group $\piet((X\smallsetminus\Delta)_{\ov k})$ is finite.
\end{proposition}
\begin{proof}
 One has $\piet((X\smallsetminus\Delta)_{\ov k})\simeq\piet((X\smallsetminus\Delta)_\CC)$ \cite{landesman}, and the latter is the profinite completion of the topological fundamental group $\pitop((X\smallsetminus\Delta)_\CC)$. It suffices to show that $\pitop((X\smallsetminus\Delta)_\CC)$ is finite.

 Because the over-exceptional lattice is a root lattice, we can contract the curves in $\Delta$ to canonical singularities and obtain a normal surface $X_0$. The minimal resolution $X\to X_0$ restricts to an isomorphism from $X\smallsetminus \Delta$ to the smooth locus $X_0^\mathrm{sm}$ of $X_0$. By a theorem of Catanese--Keum--Oguiso \cite[Theorem~A]{catanese.oguiso.keum} (which proved a conjecture of De-Qi Zhang), 
 \begin{enumerate}
  \item  $\pitop((X_0^\mathrm{sm})_\CC)$ is finite, or
  \item there exists a two-dimensional complex torus $T$ and a finite morphism $f:T\to(X_0)_\CC$ (as complex analytic space) which is unramified over $(X_0^\mathrm{sm})_\CC$. 
 \end{enumerate}
We assume (2) and derive a contradiction.

It follows from the Riemann Existence Theorem \cite[XII, Théorème~5.1]{SGA1}, that $T$ is in fact an abelian surface and $f$ is a morphism of algebraic varieties. Moreover, the proof of \cite[Theorem~A]{catanese.oguiso.keum} gives that $f$ is realised as the normalisation of the fibre product of an elliptic fibration $(X_0)_\CC\to\PP^1_\CC$ with a finite Galois cover $C\to\PP^1_\CC$ and thus it is itself a Galois cover. Denote the Galois group of $f$ by $G$. Because $T$ is minimal, the birational action of $G$ on $T$ is in fact regular.

We are thus in the same situation as described by Garbagnati in \cite{garbagnati}: we have a finite group of automorphisms $G$ of the abelian surface $T$ such that the minimal model $X$ of $T/G$ is a K3 surface. Let $F_G$ be the sublattice of $\Pic(X)$ spanned by the exceptional curves in the resolution $X \to T/G$. Note that the union of these curves has to lie in $\Delta$, since the restriction of $f$ to $(X_0)^{\mathrm{sm}} \cong X \smallsetminus \Delta$ is étale. Thus $F_G$ is a sublattice of $E'(X)$. A byproduct of Garbagnati's classification \cite[Proposition~4.3]{garbagnati} is that either $\rk F_G\geq 18$ or $F_G=A_1^{\oplus 16}$. The former is impossible by \Cref{lem:bound}. The latter is impossible because this would mean that $X$ is a Kummer surface and we showed in \cite[Corollary A, Theorem C.(3)]{gvirtz.mezzedimi.hilbert} that $(X\smallsetminus\Delta)_\CC$ is simply connected for all Kummer surfaces.
\end{proof}

\begin{proposition}\label{prop:cover}
 Let $X$ be an elliptic K3 surface over $\ov k$. For every large enough prime $p$, there exist a K3 surface $Z$ over $\ov k$ admitting at least two distinct elliptic fibrations such that $\rk E'(Z) \leq 3$ and a dominant rational map $\varphi\colon Z\dashrightarrow X$ of degree $p^2$.
\end{proposition}
\begin{proof}
 Fix an elliptic pencil $|F|$ on $X$ and denote by $J=JX$ the corresponding Jacobian fibration. Consider the class $[X]$ in $\Sh(J)\simeq\Br(J)$ and its order $d$. Since $\Sh(J)$ is divisible, for every prime $p$ we can find a K3 surface $Z$ such that $p[Z]=[X]$ and $[Z]$ has order $pd$ in $\Sh(J)$. Furthermore, the natural map $\varphi\colon Z\dashrightarrow X$ has degree $p^2$.
 
 We have $\rho(Z)=\rho(X)>1$ and by \cite[\S11, (4.5)]{huybrechts.k3.surfaces}, \[\disc(\Pic(Z))=p^2d^2\disc(J)=p^2\disc(X).\] In particular, the lattice $\Pic(Z)$ never hits an isomorphism class twice as $p$ varies over all primes.
 
 If $\rho(X)>2$, then by \cite[Theorem~5.1]{nikulin.elliptic} there are only finitely many lattices $L$ of rank $\rho(X)$ such that there exists a K3 surface $Y$ with a unique elliptic fibration and $\Pic(Y)\simeq L$. Hence the surface $Z$ admits two elliptic fibrations when $p$ is sufficiently large. Moreover, by enlarging $p$ further, one can arrange that $\rk E'(Z) \leq 3$. This is obvious when $\rho(Z)\leq 5$ from \Cref{lem:bound}. When $\rho(Z)\geq 6$, one knows from \cite[Remark~5.9(2)]{gvirtz.mezzedimi.hilbert} that $E'(Z)=0$ as long as $\Pic(Z)$ lies outside a finite list $\cS$ of lattices.
 
 If $\rho(X)=2$, then $E'(Z)=0$ as soon as $Z$ admits a second elliptic fibration (again by \Cref{lem:bound}). Hence it suffices to consider $Z$ of Picard rank $2$ with a unique elliptic fibration. By \cite[\S4.7]{van.geemen.remarks}, such $Z$ has a Picard lattice of the form
 $$\begin{pmatrix}
    0 & n\\
    n &-2
\end{pmatrix}$$
where $n$ is the order of $[Z]$ in $\Sh(J)$ (i.e.\ equal to $pd$).

When $p$ is sufficiently large and coprime to $d$, there exists a positive integer $m$ coprime to $p$ such that $m\equiv 1\pmod d$ and $m^2\not\equiv 1\pmod n$. Multiplying $[Z]$ by $m$ produces a K3 surface $Z'$ with Picard lattice of the form
 $$\begin{pmatrix}
    0 & n\\
    n &-2m^2
\end{pmatrix}.$$
(cf. \cite[Lemma~3.8]{shinder.zhang}).
Because $-2m^2\not\equiv-2\pmod n$, $Z'$ admits at least two elliptic fibrations. Furthermore, \[p[Z']=pm[Z]=m[X]=[X]\in\Sh(J).\] Replacing $Z$ by $Z'$ finishes the proof.
\end{proof}

\begin{proof}[Proof of the main theorem]
 By \Cref{prop:piet}, the fundamental group $\piet(X\smallsetminus \Delta)$ is finite. By \Cref{prop:cover}, there exists a K3 surface $Z$ over $\ov k$ with two distinct elliptic fibrations, $\rk E'(Z)\leq 3$ and a rational map $\varphi\colon Z\dashrightarrow X$ of degree $p^2$ coprime to $|\piet(X\smallsetminus \Delta)|$.
 
 After a finite extension of $k$, we can assume that $Z$ and $\varphi$ are defined over $k$, that $\Pic(Z)\simeq\Pic(Z_{\ov k})$ and that $Z(k)$ is Zariski dense in $Z$ \cite{bogomolov.tschinkel.density}. By \cite[Props. 4.1 and 4.6]{gvirtz.mezzedimi.hilbert} and \cite[Theorem~4.3]{keum.zhang.fundamental}, the condition $\rk E'(Z)\leq 3$ implies that $Z$ has the Hilbert Property over $k$.
 
 We argue that the Hilbert Property of $Z$ descends to $X$, that is $X$ also has the Hilbert Property over $k$.
 
 Assume by contradiction that there exists a finite collection of finite, surjective morphisms $\pi_i\colon Y_i\to X$, $i=1,\dots,r$, from integral, normal varieties $Y_i$ over $k$ such that
 \begin{equation}\label{eq:HP} X(k) = C(k) \cup \bigcup_i{\pi_i(Y_i(k))}\end{equation}
 for a proper closed subvariety $C\subset X$.
 
 By \cite[Proof of Theorem~1.1]{demeio.elliptic.fibrations}, (\ref{eq:HP}) still holds verbatim after we restrict to the subcollection $I\subset\{1,\dots,r\}$ of those $\pi_i$ that are unramified outside $\Delta$, i.e.\ $\pi_i\in\piet(X\smallsetminus\Delta)$ and hence $\deg(\pi_i)$ is coprime to $p$. The coprimality implies that the tensor product of function fields $k(Y_i)\otimes_{k(X)} k(Z)$ is a field (see \cite[Proposition~2.1]{cohn}) and the base change
 \[\begin{tikzcd}
    Y_i' \arrow[dashed]{r} \arrow[dashed]{d}{\pi_i'} &Y_i \arrow{d}{\pi_i}\\
    Z \arrow[dashed]{r}{\varphi} &X
\end{tikzcd}\]
is integral for all $i\in I$.

If $P\notin \varphi^{-1}(C)(k)$, then $\varphi(P)\in\pi_i(Y_i(k))$ for some $i\in I$, hence $P\in\pi_i'(Y_i'(k))$. In particular
$$Z(k)\subseteq \varphi^{-1}(C) \cup \bigcup_i {\pi_i'(Y_i'(k))},$$
contradicting the fact that $Z$ has the Hilbert Property over $k$.
\end{proof}

\begin{remark}\label{rem:bounded}
    Keeping track of the required extension $k'/k$, one argues as follows that there is a (non-explicit) uniform bound on $[k':k]$ that does not depend on $|F|$, $X$ or $k$.

    By \cite[Theorem~1.3]{lai.nakahara} $\Pic(X_{k_1})\simeq\Pic(X_{\ov k})$ for an extension $k_1/k$ of uniformly bounded degree.
    
    Moreover, the set of lattices $\cS$ is finite, hence it has an element with largest discriminant $d$. Likewise, the order of $\piet(X\smallsetminus\Delta)$ is also uniformly bounded since $\piet(X\smallsetminus\Delta)$ acts symplectically (see \cite[Proposition~5.1]{garbagnati}) and faithfully on a K3 surface and consequently has to embed into one of the $11$ groups listed in \cite[(0.4)]{mukai}. Thus the prime $p$ in \Cref{prop:cover} can be chosen uniformly.
    
    Write $\rho\coloneqq\rho(X_{\ov k})$. There is a uniformly bounded extension $k_2/k_1$ such that
    \[\Br(X_{\ov k})[p]^{\Gal(\ov k/k_2)}\simeq \Br(X_{\ov k})[p]\simeq (\ZZ/p\ZZ)^{22-\rho}.\]
    Let $\partial\colon \Br(X_{\ov k})[p]\to\RH^2(\Gal(\ov k/k_2),\Pic(X_{\ov k}))$ be the connecting homomorphism from the Hochschild--Serre spectral sequence (see \cite[Theorem~1.3]{ct.skoro}).
    A class $\alpha$ in $\Br(X_{\ov k})[p]$ comes from $\Br(X_{k_2})$ if and only if its image $\partial\alpha$ in
    \begin{align*}
        \RH^2(\Gal(\ov k/k_2),\Pic(X_{\ov k}))[p]
        \simeq\RH^1(\Gal(\ov k/k_2),(\QQ/\ZZ)^\rho)[p]\simeq\Hom(\Gal(\ov k/k_2),(\ZZ/p\ZZ)^\rho)
    \end{align*}
    is zero. Taking $k_3/k_2$ defined by $\ker(\partial\alpha)\subset\Gal(\ov k/k_2)$, one sees that $\alpha$ comes from $\Br(X_{k_3})$.
    
    It follows that $\phi:Z\dashrightarrow X$ can be defined over $k_3$. Now a uniformly bounded extension $k'/k_3$ as in \cite[Remark~6.5]{gvirtz.mezzedimi.hilbert} ensures that $Z_{k'}$, hence $X_{k'}$, has the Hilbert Property.
\end{remark}

\section{K3 surfaces with a unique elliptic fibration} \label{sec:unique.elliptic.fibration}

Throughout this section, let $X$ be a K3 surface over an algebraically closed field $\ov k$ of characteristic zero. The goal is to classify K3 surfaces with a unique elliptic fibration. 

The elliptic fibrations on $X$ correspond to primitive, isotropic, nef vectors in $\Pic(X)$, so the property of having a unique elliptic fibration depends only on the Picard lattice of $X$. Since K3 surfaces of Picard rank $1$ admit no elliptic fibration, and the classification of K3 surfaces of Picard rank $2$ admitting a unique elliptic fibration is well-known (see e.g. \cite[§4.7]{van.geemen.remarks}), we will mainly focus on the case of Picard rank $\rho \ge 3$.

We note that, if $X$ admits a unique elliptic fibration $|F|$, then all automorphisms of $X$ preserve the fibration $|F|$. In particular $X$ has \emph{zero entropy} (cf. \cite[Definition~5.1]{brandhorst.mezzedimi.borcherds}), and by \cite[Theorem~1.1]{yu.entropy} or  \cite[Theorem~5.3]{brandhorst.mezzedimi.borcherds}, together with the classification of K3 surfaces with finite automorphism group by Nikulin and Vinberg, we have a complete classification of Picard lattices of K3 surfaces of zero entropy. 

We will now present three criteria to decide whether a given K3 surface of zero entropy admits a unique elliptic fibration. We will then go through the known list of Picard lattices of K3 surfaces of zero entropy in \Cref{thm:K3.surfaces.unique.elliptic.fibration} to complete the classification. We refer to \cite{nikulin.bilinear} (or \cite[Section~2]{gvirtz.mezzedimi.hilbert}) for the basic definitions and properties of lattice theory.

\begin{proposition} \label{prop:unique.elliptic.fibration.overlattice.direct.summand}
Let $X$ and $Y$ be K3 surfaces, and assume that $\Pic(Y)$ is an overlattice of $\Pic(X)$. If $Y$ admits two distinct elliptic fibrations, then $X$ admits two distinct elliptic fibrations as well.
\end{proposition}
\begin{proof}
It suffices to show that the nef cone of $Y$ is a subcone of the nef cone of $X$. This is well-known, but let us briefly sketch a proof.

Let $A\in \Pic(X)$ be ample. If $\iota: \Pic(X)\hookrightarrow \Pic(Y)$ is an embedding, after applying an element of the Weyl group of $\Pic(Y)$ we may assume that $\iota(A)$ is ample on $Y$. Since the nef cone is dual to the effective cone, it suffices to show that $\iota$ induces an inclusion $\Eff(X)\subseteq \Eff(Y)$. For any class $C\in \Pic(X)$ of an integral curve on $X$, we have $C^2 \ge -2$ by the genus formula, and therefore by Riemann-Roch the image $\iota(C)\in \Pic(Y)$ is either effective or anti-effective. By assumption $\iota(A)$ is ample, and $\iota(A).\iota(C) = A.C>0$, so $\iota(C)$ is effective, as claimed.
\end{proof}

\begin{remark} \label{rk:findellipticfibration}
If $|F|$ is an elliptic fibration with a section on a K3 surface $X$, then $F\in \Pic(X)$ induces an orthogonal decomposition $\Pic(X)\cong U \oplus W$, where $U$ is generated by $F$ and the zero section, and $W\cong F^\perp / \langle F \rangle$.

Conversely, let $X$ be a K3 surface and assume that $\Pic(X)\cong U \oplus W$ for a negative definite lattice $W$. If $v$ is any primitive isotropic vector in $U$, up to the action of the Weyl group $W^{(2)}(\Pic(X))$ we may assume that $v$ is nef, and therefore it induces an elliptic fibration $|F|$ on $X$ such that $F^\perp / \langle F \rangle \cong W$. By \cite[Remark~11.1.4]{huybrechts.k3.surfaces}, the elliptic fibration has a section.
\end{remark}

For a definite lattice $W$, recall that the \emph{genus} of $W$ is the set of isometry classes of definite lattices $W'$ such that $U\oplus W \cong U\oplus  W'$.

\begin{proposition} \label{prop:conditions.X.unique.elliptic.fibration}
Let $X$ be a K3 surface of zero entropy, and assume $\Pic(X)\cong U\oplus W$. If $X$ admits a unique elliptic fibration, then:
\begin{enumerate}
    \item The lattice $W$ is unique in its genus;
    \item The lattice $W$ has no nontrivial overlattice.
\end{enumerate}
Moreover, if the automorphism group of $X$ is infinite, then $X$ admits a unique elliptic fibration if and only if (1) and (2) hold simultaneously.
\end{proposition}
\begin{proof}
By \Cref{rk:findellipticfibration} there exists an elliptic fibration $|F|$ with a section on $X$ such that $F^\perp / \langle F \rangle \cong W$.
Assume first that the lattice $W$ does not satisfy one of the two conditions in the statement. If $W$ is not unique in its genus, there exists a lattice $W'$, not isometric to $W$, such that $\Pic(X)\cong U \oplus W'$. By \Cref{rk:findellipticfibration} there is an elliptic fibration $|F'|$ on $X$ such that $F'^\perp / \langle F' \rangle \cong W'$. Since $W$ and $W'$ are not isometric, $|F|$ and $|F'|$ are distinct.

If instead $W$ has a nontrivial overlattice, there exists an integer $n>1$ and a nontrivial primitive vector $w\in W$ such that $W[w/n]$ is an even integral lattice. In particular $w^2 = -2k n^2$ for some $k\in \ZZ$. Choose any vector $u\in U$ of norm $2k$; then the vector $v=nu+w\in \Pic(X)$ is primitive, isotropic and it has divisibility $v.\Pic(X) \ge n$ by construction. By Riemann--Roch $v$ is effective up to a sign, and there exists an element $g$ in the Weyl group $W^{(2)}(\Pic(X))$ such that $F'\coloneqq g(v)\in \Pic(X)$ is nef.
Since the Weyl group acts as a group of isometries, $F'\in \Pic(X)$ has divisibility $\ge n>1$ as well and therefore it induces an elliptic fibration on $X$ without sections. In particular $|F|$ and $|F'|$ are distinct.

For the converse implication, assume that the automorphism group $\Aut(X)$ is infinite, and that (1) and (2) are satisfied. In particular $\rho(X)\ge 3$, since elliptic K3 surfaces of Picard rank $2$ have a finite automorphism group \cite[Example~15.2.11.(ii)]{huybrechts.k3.surfaces}. By \cite[Theorems~3.1.1~and~4.1.1]{nikulin.factor.groups} there is a lattice in the genus of $W$ that is not a \emph{root overlattice} (i.e., an overlattice of a root lattice), and since by assumption $W$ is unique in its genus, then $W$ itself is not a root overlattice.
Therefore the elliptic fibration $|F|$ has infinitely many sections by \Cref{prop:shioda.tate}. Let $|F'|$ be any elliptic fibration on $X$. If $|F'|$ has no section, then $F'\in \Pic(X)$ has divisibility $>1$ and therefore $\Pic(X)\cong U\oplus W$ admits a nontrivial overlattice. Since the discriminant group of $U$ is trivial, the orthogonal complement $W$ admits a nontrivial overlattice as well, a contradiction. Therefore $|F'|$ has a section and it induces a decomposition $\Pic(X)\cong U\oplus W'$ with $W'\cong {F'}^\perp / \langle F'\rangle$. Since by assumption $W'$ is unique in its genus, $W$ and $W'$ are isometric, and therefore $|F'|$ admits infinitely many sections by \Cref{prop:shioda.tate}. But $X$ has zero entropy, so by \cite[Theorem~1.6]{oguiso.entropy} $|F'|$ and $|F|$ are the same elliptic fibration. In particular $|F|$ is the only elliptic fibration on $X$.
\end{proof}

\begin{proposition} \label{prop:unique.elliptic.fibration.root.overlattice}
Let $X$ be a K3 surface, and assume that $X$ has zero entropy and an infinite automorphism group. If $X$ admits two distinct elliptic fibrations, then:
\begin{enumerate}
    \item There exists a root overlattice $N$ of rank $\rho(X)-2$ such that $\disc(\Pic(X))/\disc(N)$ is a square;
    \item If $n^2 = \disc(\Pic(X))/\disc(N)$, there exists an isotropic vector of order $n$ in the discriminant group of $\Pic(X)$.
\end{enumerate}
\end{proposition}
\begin{proof}
By \cite[Theorem~1.6]{oguiso.entropy} $X$ admits precisely one elliptic fibration $|F|$ with infinite stabilizer in $\Aut(X)$, hence the second elliptic fibration $|F'|$ on $X$ has finite stabilizer. Therefore by \Cref{prop:shioda.tate} the negative definite lattice $N\coloneqq {F'}^\perp / \langle F' \rangle$ is a root overlattice of rank $\rho(X)-2$.
Let $n\ge 1$ be the divisibility of $F'\in \Pic(X)$. The vector $F'\in \Pic(X)$ induces an isotropic vector of order $n$ in the discriminant group of $\Pic(X)$, and dividing $F'$ by $n$ gives an overlattice $L$ of $\Pic(X)$. By construction we have $L\cong U \oplus N$, so
$$\disc(\Pic(X))=n^2 \disc(L) = n^2 \disc(N).$$
\end{proof}

We are now ready to complete the classification of the Picard lattices of K3 surfaces with a unique elliptic fibration.

\begin{theorem} \label{thm:K3.surfaces.unique.elliptic.fibration}
A K3 surface $X$ over $\ov k$ admits a unique elliptic fibration if and only if its Picard lattice is isometric to
$$\begin{pmatrix}
    0 &n\\
    n &-2
\end{pmatrix}$$
or to one of the $92$ explicit lattices of rank $\ge 3$ listed in the ancillary file\footnote{or at \url{https://github.com/dgvirtz/uniqueellipticK3}.\label{footnote:link}}. In particular every K3 surface of Picard rank $\rho(X) > 12$ admits at least two distinct elliptic fibrations.
\end{theorem}
\begin{proof}
If the Picard rank of $X$ is $2$, we refer to \cite[§4.7]{van.geemen.remarks}.
Furthermore, if the Picard rank of $X$ is at least $13$, then $\Pic(X)\cong U\oplus W$ for a negative definite lattice $W$ by \cite[Corollary~14.3.8]{huybrechts.k3.surfaces}, and $W$ is not unique in its genus \cite{lorch.kirschmer.single.class}, so $X$ admits at least two distinct elliptic fibrations by \Cref{prop:conditions.X.unique.elliptic.fibration}. Therefore in the following we can restrict to the case $3\le \rho(X)\le 12$.
Let $X$ be a K3 surface with a unique elliptic fibration $|F|$. The automorphism group $\Aut(X)$ preserves $F\in \Pic(X)$, and therefore $X$ has zero entropy.

We assume first that $X$ has a finite automorphism group. Recall that a K3 surface with finite automorphism group has only finitely many $(-2)$-curves \cite[Corollary~8.4.7]{huybrechts.k3.surfaces}, and we refer to \cite{roulleau.atlas} and \cite{artebani.correa.roulleau} for an explicit description of the dual graphs of $(-2)$-curves on the $99$ families of K3 surfaces with finite automorphism group in rank $3\le \rho(X)\le 12$. Denote by $\Gamma$ the dual graph of $(-2)$-curves on $X$. Since $\rho(X)-2>0$, by \Cref{prop:shioda.tate} every elliptic fibration on $X$ has a reducible fiber, which consists of $(-2)$-curves, so we can count the elliptic fibrations on $X$ by classifying the extended Dynkin diagrams supported on $\Gamma$. Let us give two examples for the sake of completeness. The dual graph of $(-2)$-curves on K3 surfaces with $\Pic(X)\cong \langle 4 \rangle \oplus D_4$ (resp. $\Pic(X)\cong \langle 8 \rangle \oplus D_4$) is as follows:
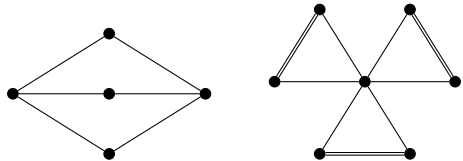
\begin{equation*} 
 \begin{tikzpicture}[scale=0.8]
    \node (A5) at (0,0) [nodal] {};
    \node (A3) at (1.6,0) [nodal] {};
    \node (A2) at (0,1) [nodal] {};
    \node (A1) at (-1.6,0) [nodal] {};
    \node (A4) at (0,-1) [nodal] {};

\draw (A1)--(A2)--(A3)--(A4)--(A1) (A1)--(A5)--(A3);
\end{tikzpicture}
\qquad
 \begin{tikzpicture}[scale=0.6]
    \node (A5) at (2,0) [nodal] {};
    \node (A3) at (1,-1.6) [nodal] {};
    \node (A2) at (1,1.6) [nodal] {};
    \node (A1) at (-2,0) [nodal] {};
    \node (A4) at (-1,1.6) [nodal] {};
    \node (A6) at (-1,-1.6) [nodal] {};
    \node (A7) at (0,0) [nodal] {};

\draw (A1)--(A7)--(A5) (A4)--(A7)--(A3) (A2)--(A7)--(A6);
\draw[double] (A1)--(A4) (A2)--(A5) (A3)--(A6);
\end{tikzpicture}
\end{equation*}
In the first graph there are three extended Dynkin diagrams, two of type $\widetilde{A}_2$ and one of type $\widetilde{A}_3$. They correspond to elliptic fibrations $|F_1|$, $|F_2|$ and $|F_3|$ such that $F_i.F_j=2$ for $i\ne j \in \{1,2,3\}$, so $X$ admits precisely $3$ elliptic fibrations.
In the second graph there are three extended Dynkin diagrams, all of type $\widetilde{A}_1$. Since they are disjoint, they correspond to reducible fibers of the same elliptic fibration. In particular $X$ admits a unique elliptic fibration.

Assume now that $X$ has infinite automorphism group. Then $\Pic(X)$ is one of the $176$ lattices of rank $3\le \rho(X) \le 12$ in \cite[Corollary~1.5]{brandhorst.mezzedimi.borcherds} or \cite[Theorem~1.1]{yu.entropy}.
By \Cref{prop:conditions.X.unique.elliptic.fibration} we have necessary and sufficient conditions for $X$ to admit a unique elliptic fibration if $\Pic(X)$ contains a copy of the hyperbolic plane $U$. Hence in this case a quick computation in \texttt{Magma} \cite{magma} (see the link in footnote \ref{footnote:link} for code) singles out the Picard lattices of K3 surfaces with a unique elliptic fibration.
If instead $\Pic(X)$ does not contain a copy of $U$, let $|F|$ be an elliptic fibration on $X$ (necessarily without section). If $L$ is the lattice obtained from $\Pic(X)$ by dividing $F$ by its divisibility $\mathrm{div}(F)$, by construction $L$ contains a copy of $U$. By the previous analysis we know whether any K3 surface $Y$ with $\Pic(Y)\cong L$ admits two distinct elliptic fibrations. If it does, then $X$ admits two elliptic fibrations as well by \Cref{prop:unique.elliptic.fibration.overlattice.direct.summand}, and we can exclude the Picard lattice $\Pic(X)$.

After applying the criteria above, we remain with $6$ lattices in rank $3$, $3$ lattices in rank $4$ and $1$ lattice in rank $5$.
We claim that for all these Picard lattices $\Pic(X)$, the K3 surface $X$ admits a unique elliptic fibration.

For all these $10$ lattices, except
$$L_1=\begin{pmatrix}
    0 &2 &0 &0\\ 2 &-2 &0 &0\\ 0 &0 &-4 &-2\\ 0 &0 &-2 &-4
\end{pmatrix},\qquad
L_2=
\begin{pmatrix}
    0 &2 &0 &0 &0\\ 2 &-2 &1 &0 &0\\ 0 &1 &-2 &0 &0\\ 0 &0 &0 &-4 &-2\\ 0 &0 &0 &-2 &-4
\end{pmatrix},$$
it is straightforward to check that there exists no root overlattice $N$ of rank $\rho(X)-2$ such that $\disc(\Pic(X))/\disc(N)$ is a square. Therefore the corresponding K3 surfaces admit a unique elliptic fibration by part (1) of \Cref{prop:unique.elliptic.fibration.root.overlattice}.

On the other hand, for $L_1$ (resp. $L_2$) the root overlattice $N=A_2$ (resp. $N=A_2\oplus A_1$) is such that $\disc(L_1)/\disc(N) = \disc(L_2)/\disc(N) = 4^2$.
A simple computation shows that there is no vector of order $4$ in the discriminant group $A_{L_1}$, while the only vectors of order $4$ in $A_{L_2}$ have norm $\pm \frac{1}{2} \pmod{2\ZZ}$, so the K3 surfaces with $\Pic(X)\cong L_1$ or $L_2$ admit a unique elliptic fibration by part (2) of \Cref{prop:unique.elliptic.fibration.root.overlattice}. 
\end{proof}

\begin{remark} \label{rk:unique.elliptic.fibration.positive.char}
The classification in \Cref{thm:K3.surfaces.unique.elliptic.fibration} is also valid over an algebraically closed field of arbitrary characteristic. 
Indeed non-supersingular K3 surfaces can be lifted to characteristic $0$ together with their Picard lattice, and since the property of having a unique elliptic fibration only depends on the Picard lattice, the classification in \Cref{thm:K3.surfaces.unique.elliptic.fibration} immediately extends to non-supersingular K3 surfaces.
On the other hand, supersingular K3 surfaces are known to always admit an automorphism of positive entropy \cite[Corollary~5.4]{brandhorst.mezzedimi.borcherds}, so they admit at least two elliptic fibrations. 
\end{remark}

\printbibliography

\end{document}